\documentclass[preprint,8pt]{elsarticle}

\usepackage{caption}
\usepackage{subcaption}
\usepackage[]{amsmath}
\usepackage{amsthm}
\usepackage{amsfonts}
\usepackage{xcolor}
\usepackage{soul}

\usepackage{geometry}
\newcommand{\ds}{\displaystyle}

\theoremstyle{plain}
\newtheorem{theorem}{Theorem}[section] \newtheorem*{theorem*}{Theorem}
 \newtheorem*{proposition*}{Proposition}
 \newtheorem*{lemma*}{Lemma}
 \newtheorem*{corollary*}{Corollary}
\newtheorem{remark}[theorem]{Remark} \newtheorem*{remark*}{Remark}

\theoremstyle{definition}
 \newtheorem*{definition*}{Definition}
\newtheorem{example}[theorem]{Example} \newtheorem*{example*}{Example}
 \newtheorem{assumption}[theorem]{Assumption}

\numberwithin{equation}{section}

\newcommand{\bR}{\mathbb{R}}
\newcommand{\bN}{\mathbb{N}}
\newcommand{\bE}{\mathbb{E}}
\newcommand{\bP}{\mathbb{P}}

\newcommand{\cA}{\mathcal{A}}

\newcommand{\floor}[1]{\lfloor{#1}\rfloor} %sofarsogood

\begin{document}

\begin{frontmatter}

\title{Bayesian Approach for Parameter Estimation of Continuous-Time Stochastic Volatility Models using Fourier Transform Methods}
\author[label1]{Milan Merkle}
\address[label1]{School of Electrical Engineering, University of Belgrade, Serbia.}
\author[label2]{Yuri F. Saporito}
\author[label2]{Rodrigo S. Targino}
\address[label2]{Escola de Matem\'atica Aplicada (EMAp), Funda\c{c}\~ao Getulio Vargas (FGV), Rio de Janeiro, Brazil}

%\address[label3]{FGV-RJ/EMAp}

\address{}

%\maketitle

\begin{abstract}

We propose a two stage procedure for the estimation of the parameters of a fairly general, continuous-time stochastic volatility.
An important ingredient of the proposed method is the Cuchiero-Teichmann volatility estimator, which is based on Fourier transforms
and provides a continuous time estimate of the latent process. This estimate is then used to construct an approximate
likelihood for the parameters of interest, whose restrictions are taken into account through prior distributions.
The procedure is shown to be highly successful for constructing the posterior distribution of the parameters of a Heston
model, while limited success is achieved when applied to the highly parametrized exponential-Ornstein-Uhlenbeck.

\end{abstract}


\end{frontmatter}
%\epigraph{\it{If you spend a lot of money on communication, you want it to be fruitful}{--- \textup{}, }}

\section{Introduction}

%{\sl  References about history and methods mentioned in this part are taken from  \cite{forecast2008}, \cite{cute2015+}, \cite{karodell2010+},
%\cite{mama2002a+}, \cite{mama2009+} and \cite{twoscales+}.}

\medskip

For a given filtered probability space $(\Omega, \cA, (\cA_t)_{t\in [0,T]}, \bP)$ with a Brownian motion $W$,
a process $X$  is called standard It\^{o} process if it allows the integral representation
\begin{equation}
\label{sip}
X_t  = x_0 +\int_0^t \mu_s d s + \int_0^t \sigma_s d W_s,\quad 0\leq t \leq T,
\end{equation}
where $\mu$ (drift)  and $\sigma>0$ (diffusion coefficient) are adapted measurable processes that satisfy certain conditions
to ensure existence of integrals (see, for example, \cite{rogwil}). The  diffusion coefficient  of an It\^{o} processes,
known as volatility in financial models,
has been studied intensely due to its immense significance  in applications. If the volatility effectively depends on $\omega\in \Omega$ via
another standard It\^o process $V_t$ with
\begin{equation}
\label{sipv}
V_t = v_0 +\int_0^t \mu^V_s d s + \int_0^t \sigma^V_s d W_s^V,\quad 0\leq t \leq T,
\end{equation}
where $W^V$ is a Brownian motion, we say that the process (\ref{sip}) is equipped with stochastic volatility.
From the theory of It\^{o} processes it follows that $\int_0^t \sigma^2_s d s$ (the integrated realized variance) is equal to quadratic
variation of  $X$ in the interval $[0,t]$, and can be estimated  from a discretely observed path of $X$ as  the sum of squares of increments
on $[0,t]$.  Then  $\sigma^2 $ as a function of $t$ (instantaneous variance) can be recovered from the estimate by differentiation.
However, the implementation with high frequency samples from real time series of asset prices showed certain drawbacks due to the market noise, as
explained and analyzed in Zhou \cite{Zhou1996} and  Zhang et al. \cite{twoscales+}, see also  A{\"{i}}t-Sahalia et al. \cite{forecast2008} for a large-scale
simulation study of the integrated variance estimator from \cite{twoscales+}.
Malliavin and Mancino (\cite{mama2002a+}, \cite{mama2009+}) offered a non-parametric method for a direct estimation of instantaneous variance, based on
Fouried transform.  Using similar ideas, Cuchiero and Teichmann \cite{cute2015+} proposed a robust estimator, which allows processes with jumps.
 Both Malliavin-Mancino and Cuchiero-Teichmann estimators are applicable in a multidimensional setup. There are several other instantaneous variance
 estimators (see \cite{cute2015+} and \cite{mama2009+} for references).

Diffusion processes usually involve some parameters, which, in practice, would have to be estimated.
There is a voluminous literature on Bayesian inference in this context (for example, \cite{elchsh2001+}, \cite{Erak2001+},
\cite{karodell2010+},  \cite{sore2004}). In this paper we consider a general family of continuous-time stochastic volatility models with five parameters. With particular specifications, this family covers several well known models, such as the Heston and exponential-Ornstein-Uhlenbeck (exp-OU) models. We propose a two-stage method for parameter estimation. At  the first stage, we recover  the realized volatility process using Cucheiro-Teichmann estimator, and  at the second stage, conditionally  on the recovered volatility,
we estimate the parameters using Bayesian technique.

\section{A general stochastic volatility model: properties and volatility estimation}

In this section we present the model and its properties, together with the one-dimensional version of the Cucheiro-Teichmann procedure \cite{cute2015+}.
The Bayesian parameter estimation will be discussed in the Section \ref{sec:estimation}.

\subsection{The model and its properties}

\label{modelp}

Let $W$ and $W^{\perp}$ be independent standard Brownian motions on  the fixed probability space $(\Omega, \cA, \bP)$.  Consider the following stochastic volatility model in the differential form,  as a system of stochastic differential equations (SDEs):
\begin{align}\label{eq:sv_model}
\left\{\begin{array}{l}
dX_t = \ds \mu dt + f(V_t) dW_t,\\ \\
dV_t = \kappa(m - V_t) dt + \xi g(V_t) dW_t^V,\\ \\
W_t^V = \rho W_t + \sqrt{1-\rho^2} W_t^{\perp},
\end{array}\right.
\end{align}
where $\rho \in [-1,1]$. This parameter reflects a correlation
between $W$ and $W^V$. Other parameters of the model under consideration and their ranges are $\mu \in \bR$, $m\in \bR$, $\kappa >0$, $\xi>0$. The usual names in
the financial literature are exhibited in Table \ref{table:param}.

For example, $f(v) = g(v) = \sqrt{v}$ gives us the Heston model, see \cite{heston93}, and $f(v) = e^v, g(v) = 1$,
the exponential-Ornstein Uhlenbeck (exp-OU) model, see \cite{pesima2008+}. Different models (that consider different drifts for the volatility) could be considered with little modification of the technique presented below.
A class of models with  $f=Cg$, where $C>0$ is a specified known constant, is frequently used in applications. We call these models \textit{equi-volatility models};
the Heston model is one example of such model. The Inverse Gamma model of \cite{lalezu2016+} is also a member of this class.

The additional necessary assumptions on the model are stated as follows.

\begin{assumption}\label{ass:f_over_g}
$(i)$ $f$ and $g$ are positive functions on the support of the volatility process $V$, and the function $f$ is strictly monotone;
$(ii)$ The SDE (\ref{eq:sv_model}) has a unique strong solution, with a positive volatility process $V$.
\end{assumption}

The following theorem states that the increments $X_t - X_s$, given the volatility path, are normally distributed with mean and variance that can be
explicitly computed.
This result is  a paramount in our Bayesian estimation procedure described
in Section \ref{sec:estimation}.

\begin{theorem}\label{thm:distribution}
Under Assumption \ref{ass:f_over_g}, given the volatility path $(V_u)_{u \in [0,T]}$, the increments $X_t -X_s$, $0\leq s<t\leq T$,
are independent  over disjoint intervals and
normally distributed with mean and variance as below:
$$X_t - X_s \ | \ (V_u)_{u \in [0,T]} \sim N\left( \int_s^t \phi(V_u; \theta) du + \frac{\rho}{\xi} \int_s^t \frac{f(V_u)}{g(V_u)} dV_u, \ \ (1 - \rho^2)\int_s^t f^2(V_u) du\right)$$
where $\theta = (\mu, \kappa, m, \rho, \xi)$ and
\begin{align}\label{eq:phi}
\phi(v; \theta) &= \mu - \frac{\kappa \rho m}{\xi} \frac{f(v)}{g(v)} + \frac{\kappa \rho}{\xi} \frac{f(v)}{g(v)} v.
\end{align}
\end{theorem}

\begin{proof} From (\ref{eq:sv_model}) we find that
$$\int_s^t f(V_u) dW_u = \rho\int_s^t f(V_u) dW_u^V + \sqrt{1 - \rho^2}\int_s^t f(V_u) dW_u^{\perp},$$
where $W^{\perp}$ is a Brownian motion independent of $W^V$. Then it follows that
$$\left. \int_s^t f(V_u) dW_u^{\perp} \ \right| \ (V_u)_{u \in [s,t]} \sim N\left(0, \int_s^t f^2(V_u) du \right).$$
Moreover, from the SDE describing the dynamics of the process $V$, we find that
\begin{align*}
\int_s^t f(V_u) dW_u^V = \frac{1}{\xi} \int_s^t \frac{f(V_u)}{g(V_u)} dV_u - \frac{1}{\xi} \int_s^t \frac{f(V_u)}{g(V_u)} \kappa(m - V_u) du
\end{align*}
and this yields the result.
\end{proof}

In Table \ref{table:phi}, we specify the function $\phi$ of two particular models.

\begin{table}[h!]
\centering
\begin{tabular}{lcc}
\hline
Model & Specification & $\phi(v;\theta)$ \\
\hline
\hline \\[-5pt]
Equi-Volatility & $f = g$ & $\displaystyle \mu - \frac{\kappa \rho m}{\xi}  + \frac{\kappa \rho}{\xi} v$ \\[15pt]
Exp-OU & $f(v) = e^v$ and $g \equiv 1$ & $\displaystyle  \mu - \frac{\rho \kappa m}{\xi} e^v + \frac{\rho \kappa }{\xi}v  e^v$ \\[10pt]
\hline
\end{tabular}
\caption{The function $\phi$ for the equi-volatility and Exp-OU models.}
\label{table:phi}
\end{table}

\subsection{Pathwise covariance estimation of Cuchiero and Teichmann}\label{sec:teichmann}

The estimation procedure for the parameters of the SV model discussed in the previous subsection is built on the following
estimation method of the hidden volatility process, which has been proposed and shown to be consistent in \cite{cute2015+}.
We will now describe this method. More generally, we assume that $X$ follows the dynamics
\begin{align}
dX_t = \mu_tdt + f(V_t) dW_t,
\end{align}
where $W$ is an one-dimensional Brownian motion, $\mu$ is a locally bounded process and $V$ is a continuous stochastic process.
Fix a time horizon $T > 0$ and define $s_m^n = m/n$, for $m = 0, \ldots, \floor{nT}$.
It is assumed that we observe the process $X$ at times $\{s_m^n\}_{m=0}^{\floor{nT}}$.\\

Given a continuous function $h$ with at most polynomial growth{\color{red},} the \textit{Cuchiero-Teichmann estimator} of
the instantaneous variance is given by
\begin{align}\label{eq:pathwise_estimator}
\widehat{V}_t^{n,N} = f^{-1}\left(\sqrt{\rho_h^{-1}\left(\widehat{\rho_h(V)}_t^{n,N} \right)}\right),
\end{align}
where
\begin{align}
\widehat{\rho_h(V)}_t^{n,N} &= \frac{1}{T} \sum_{k=-N}^N \left(1 - \frac{|k|}{N} \right)e^{i \frac{2\pi}{T} k t} G(X, h, k)_T^n, \\
G(X, h, k)_T^n &= \frac{1}{n} \sum_{m=1}^{\floor{nT}} e^{-i \frac{2\pi}{T} k s_{m-1}^n} h(\sqrt{n} (X_{s_m^n} - X_{s_{m-1}^n})),
\end{align}
and the function $x\mapsto \rho_h(x)$ is defined as  $\rho_h(x) = \bE[h(Z)]$, for $Z \sim N(0, x)$.

\begin{example}\label{exam:ct_estimator}
For instance, we may choose $h(x) = \cos(x)$, which gives us $\rho_h(x) = e^{-\frac{1}{2}x}$ and $\rho_h^{-1}(x) = -2\log x$.

The consistency of the estimator is guaranteed if $\ds \lim_{n, N \to +\infty} n/N^\gamma > 0$, for some $\gamma > 1$.
One might take $N = \floor{\sqrt{n}}$.\\

Even though the volatility estimator from (\ref{eq:pathwise_estimator}) is defined for continuous time
$t$, one could evaluate the process $\widehat{V}^{n,N}$ at times $t_k = kT/(2N+1)$, $k=0,\ldots, 2N+1$.
We choose $n$ and $N$ in such {\color{red} a} way that every $t_k$ is one of the $s_m^n$, i.e. we may consider both $X$ and
$\widehat{V}^{n,N}$ at times $t_k$.
\end{example}

\section{Bayesian estimation procedures for the parameters for stochastic volatility models}\label{sec:estimation}

In this section, we will describe the estimation procedure based on Theorem \ref{thm:distribution} and the volatility estimation of Cuchiero-Teichmann presented in Section \ref{sec:teichmann}.

For a sequence of equally spaced time steps $(t_k)_{k=1,\ldots,K}$, with $\Delta t = t_{k}-t_{k-1}$, we define the increments of the log-price process as $\Delta X_k = X_{t_{k}} -X_{t_{k-1}}$. Theorem \ref{thm:distribution} states that, conditional on the entire path of true volatility process $V$, the increments are independent and normally distributed:
\begin{equation} \label{eq:regression_general}
\Delta X_k \, | \, (V_u)_{u \in [0, T]} \stackrel{ind}{\sim} N\left(M_k(\theta), \ (1-\rho^2) F_k^2 \right)\,
\end{equation}
with $F_k^2 = \int_{t_{k-1}}^{t_k} f^2(V_u) du$ and
\begin{equation}
M_k(\theta) = \int_{t_{k-1}}^{t_k} \phi(V_u; \theta) du + \frac{\rho}{\xi} \int_{t_{k-1}}^{t_k} \frac{f(V_u)}{g(V_u)} dV_u.
\label{eq:conditional_mean}
\end{equation}
Then, under model (\ref{eq:regression_general}), the likelihood for the parameter $\theta$, conditional on the volatility process, is given by
\begin{align*}
\ell(\theta; \, \Delta \mathbf{x} \, | \, (V_u)_{u \in [0, T]}) = \big(2\pi(1-\rho^2)\big)^{-K/2} \left(\prod_{k=1}^K F_k^{-1} \right)\exp\left\{-\frac{1}{2(1-\rho^2)}\sum_{k=1}^K\left(\frac{\Delta x_k - M_k(\theta)}{F_k}\right)^2 \right\}.
\end{align*}
As the process $V$ is non-observable, in order to be able to perform parametric inference for $\theta$, we use an approximate likelihood approach (see, e.g., \cite{drovandi2015bayesian} for several simulation-based approaches in the context of Bayesian inference). For fixed $n$ and $N$, let $\widehat{\ell}_{n,N}$ denote the approximated likelihood, defined as
\begin{align}
\widehat{\ell}_{n,N}(\theta; \, \Delta \mathbf{ x}) = \big(2\pi(1-\rho^2)\big)^{-K/2} \left( \prod_{k=1}^K \widehat{F}^{-1}_k \right) \exp\left\{-\frac{1}{2(1-\rho^2)^2}\sum_{k=1}^K\left(\frac{\Delta x_k - \widehat{M_k}(\theta)}{\widehat{F}_k}\right)^2 \right\},\label{eq:approx_likelihood}
\end{align}
where the process $V$ is replaced by its Cuchiero-Teichmann estimate, $\widehat{V}$. In (\ref{eq:approx_likelihood}) all the integrals are approximated by quadrature, providing the following definitions:
\begin{align*}
\widehat{F}^2_k = f^2(\widehat{V}_{t_{k-1}}) \ds \Delta t \ \ \text{ and } \ \ \widehat{M}_k(\theta) = \phi(\widehat{V}_{t_{k-1}}; \theta) \Delta t + \frac{\rho}{\xi} \ \frac{f(\widehat{V}_{t_{k-1}})}{g(\widehat{V}_{t_{k-1}})} \Delta \widehat{V}_k.
\end{align*}
Note that, for the sake of notational simplicity, we are dropping the superscript $n, N$ from the volatility estimate and assuming that $K=2N+1$.

\begin{remark}
For fixed $(t_k)_{k=1,\ldots,K}$, under some additional regularity assumptions, one could show the convergence of the approximated likelihood $\widehat{\ell}_{n,N}$ to $\ell$, as $n$ and $N$ go to infinity satisfying $\ds \lim_{n, N \to +\infty} n/N^\gamma > 0$, for some $\gamma > 1$. Although the rigorous verification of this convergence is outside the scope of this letter, in Figure \ref{fig:qq_plot} we motivate this result by comparing the quantiles of two different normal distributions:
$$\epsilon_k = \frac{\Delta X_k - M_k(\theta)}{\sqrt{1 - \rho^2}F_k} \mbox{ and } \widehat{\epsilon}_k = \frac{\Delta X_k - \widehat{M}_k(\theta)}{\sqrt{1 - \rho^2}\widehat{F}_k}$$
Note that on both distributions the parameters are set to their true values and on the first model we condition on the true volatility path.
\end{remark}

\subsection{Parameter identifiability}

Firstly, the variance of the increments of $X$ depends only on the unknown parameter $\rho$ through its square.
This means that absolute value of $\rho$ is identifiable. See \ref{sec:identifiability}
for a more throughout discussion on the identifiability issue for parameter estimation.
Secondly, the presence of a Brownian motion in $V$ makes the increments of $F(V)$ and the
integrals of $\phi(V;\theta)$ fundamentally different, i.e. the increments of $F(V)$ have the rough behavior of the Brownian
motion and the integrals of $\phi(V;\theta)$ are of bounded variation, which gives us a strong indication that $\beta$ is identifiable.

Therefore, since $\xi > 0$, we have $\mbox{sign}(\rho) = \mbox{sign}(\beta)$, which implies $\rho = |\rho| \mbox{sign}(\beta)$ and $\xi = \rho/\beta$ are identifiable in the estimation of the general statistical
model (\ref{eq:regression_general}). In order to study the identifiability of the other parameters $\mu, \kappa$ and $m$, one would need to specify the volatility functions $f$ and $g$.

In the equi-volatility models, using the specification of $\phi$ given in Table \ref{table:phi}, the first integral in (\ref{eq:conditional_mean}) involves a constant and an integral of $V$:
$$\int_{t_{k-1}}^{t_k} \phi(V_u; \theta) du = \left(\mu - \frac{\kappa \rho m}{\xi}\right)(t_k - t_{k-1})  +
\frac{\kappa \rho}{\xi} \int_{t_{k-1}}^{t_k} V_u du$$
Since these two terms are different functions of the data, we may conclude that we can identify the coefficients of the time increment and the integral of $V$, which implies that $\kappa$ is identifiable, but $\mu$ and $m$ are not.
In the exp-OU model, following a similar reasoning, we conclude that $\mu$, $m$ and $\kappa$ are identifiable.

\begin{remark}[A drawback of frequentist inference] As the model in (\ref{eq:approx_likelihood}) is a simple linear regression it is clear that a pure likelihood-based estimation procedure could possibly estimate the variance term $1-\rho^2$ by a negative value. In order to avoid inestimability, the Bayesian procedure proposed in the next section assigns the uniform prior on the interval $[-1,1]$ to $\rho$.
\end{remark}

\subsection{The estimation procedure and implementation on selected models}

In order to construct the posterior distribution of the parameters of interest, the (approximated) likelihood in (\ref{eq:approx_likelihood}) is combined with appropriated prior distributions. Samples from the posterior distribution are generated using the following procedure:

\begin{enumerate}

\item from the observed data $X_{s_1}, \ldots, X_{s_n}$, estimate $V_{t_1}, \ldots, V_{t_N}$ using the
Cuchiero-Teichmann procedure described in Section \ref{sec:teichmann};

\item assign independent prior distributions to $\xi$ and $\rho$ respecting the restrictions that $\xi > 0$ and $\rho \in [-1,1]$;

\begin{itemize}

\item[--] in our procedure, we choose uninformative priors on (sensible) finite intervals, see Table \ref{table:posterior_heston};

\end{itemize}

\item using $(X_{t_1}, \widehat{V}_{t_1}), \ldots, (X_{t_N}, \widehat{V}_{t_N})$
and Equation (\ref{eq:approx_likelihood}), generate samples from the posterior distributions of $(\xi, \rho)$;

\begin{enumerate}

\item Under the equi-volatility models, we may also generate samples from the posterior distribution of $\kappa$,
but we cannot separate the effects of $\mu$ and $m$, see Table (\ref{table:phi}).

\item Under the Exp-OU model, we are able to generate samples from the posterior distribution of all the other parameters:
$\kappa, m$ and $\mu$.

\end{enumerate}

\end{enumerate}

Step 3 of the above procedure is performed with the aid of a Hamiltonian Monte Carlo algorithm, implemented through \texttt{R-Stan} \cite{rstan}.

\section{Numerical exercise}

The numerical experiment will consider simulated data from the Heston model (an example of equi-volatility model) and the exp-OU model (where $f$ and $g$ are different). The function $h$ for both models is defined as in the Example \ref{exam:ct_estimator}. The parameters (of the model and the numerical procedure) are described in Table \ref{table:param} and the prior distributions in Table \ref{table:posterior_heston}.

\begin{table}[h!]
\centering
\begin{tabular}{clc}
\hline
Parameter & Description & Value \\
\hline
\hline\\[-10pt]
$T$ & Time horizon & 1.0 \\
$n$ & Number of observations & $2^{19}$ \\
$N$ & C-T Frequency & $2^9$\\
$X_0$ & initial log-price & 0.0\\
$\mu$ & return rate & 0.0\\
\hline
\end{tabular}
\quad
\begin{tabular}{clc}
\hline
Parameter & Description & Value \\
\hline
\hline\\[-10pt]
$V_0$ & initial variance & 0.09\\
$\kappa$ & mean-reversion rate & 5.0\\
$m$ & long-run mean & 0.02\\
$\xi$ & vol-of-vol & 0.5\\
$\rho$ & correlation & -0.3\\
\hline
\end{tabular}
\caption{Parameters' Description and Values}
\label{table:param}
\end{table}

\begin{table}[h!]
\centering
\begin{tabular}{cl}
\multicolumn{2}{c}{}\\
\hline
Parameter & Priors \\
\hline
\hline\\[-10pt]
$\rho$ & $U[-1,1]$ \\
$\xi$ & $U[0,5]$ \\
$\kappa$ & $U[0,100]$ \\
$m$ & $U[0,1]$ \\
$\mu$ & $U[-1,1]$ \\
\hline
\end{tabular} \label{table:prior}
\quad
\begin{tabular}{ccc}
\multicolumn{3}{c}{Heston}\\
\hline
Median & 2.5\% & 97.5\% \\
\hline
\hline\\[-10pt]
-0.2890 & -0.4667 & -0.0080\\
0.3444 & 0.1034 & 0.8178\\
1.2438 & 0.0529 & 6.3202\\
\\
\\
\hline
\end{tabular}
\quad
\begin{tabular}{ccc}
\multicolumn{3}{c}{Exp-OU}\\
\hline
Median & 2.5\% & 97.5\% \\
\hline
\hline\\[-10pt]
-0.1946 & -0.4194 & 0.0082\\
0.9183 & 0.1667 & 4.7541\\
9.6601 & 0.2432 & 89.1850\\
0.2796 & 0.0075 & 0.9422\\
-0.3412 & -0.9704 & 0.8405\\
\hline
\end{tabular}
\caption{Priors and Posteriors distributions}
\label{table:posterior_heston}
\end{table}

In Figure \ref{fig:vol} the true volatility process (dark blue) and its Cuchiero-Teichmann estimate (light blue)
are presented, both for the (a) Heston and the (b) Exp-OU models. It can be seen that the estimate is able to closely
follow the true unobserved paths for both models, with the estimate for the Heston model been (at least visually) more precise.
As the estimate is \emph{offline}, in the sense that it its computed for a batch of data and needs to be recomputed when new
observation arrive, errors at the beginning/end of the estimation window should not be notoriously different.
It should also be noticed that the estimates are reliable at both low and high (absolute) volatility regimes, see Figure \ref{fig:vol}.
One important aspect, though, is the apparent lower volatility of the estimated paths, i.e., the reconstructed functions
appear to be smoother than the original ones.

As hinted by Figure \ref{fig:vol}, Figure \ref{fig:param} shows that the parameters of the Heston model are,
indeed, better estimated when compared to the exp-OU model. For both models we present all two dimensional
posterior distributions in the lower triangular plots. At these plots we also include (in blue) the real
 parameter values used in the data simulation. In the main diagonal plots we present the histogram of the
 identifiable parameters, where the blue solid line represents the real value and the dashed line its marginal
 posterior mean. Both for the Heston and the exp-OU models the correlation parameter $\rho$ is estimated remarkably well,
 with high posterior probability been assigned to its correct sign (negative on both cases). While the two remaining parameters
 of the Heston model are also very well estimated, the same does not hold true for the other exp-OU parameters, as most of them
 have uniform marginal posterior distributions.

It should be stressed that these results are consistently observed for these models, independent of the particular sample path generated, which leads us to believe that the proposed method can be a competitive alternative to the state-of-the art algorithms used for inference in continuous time stochastic volatility models.

\begin{figure}[h!]
\centering
\begin{subfigure}{.5\textwidth}
  \centering
  \includegraphics[width=\linewidth]{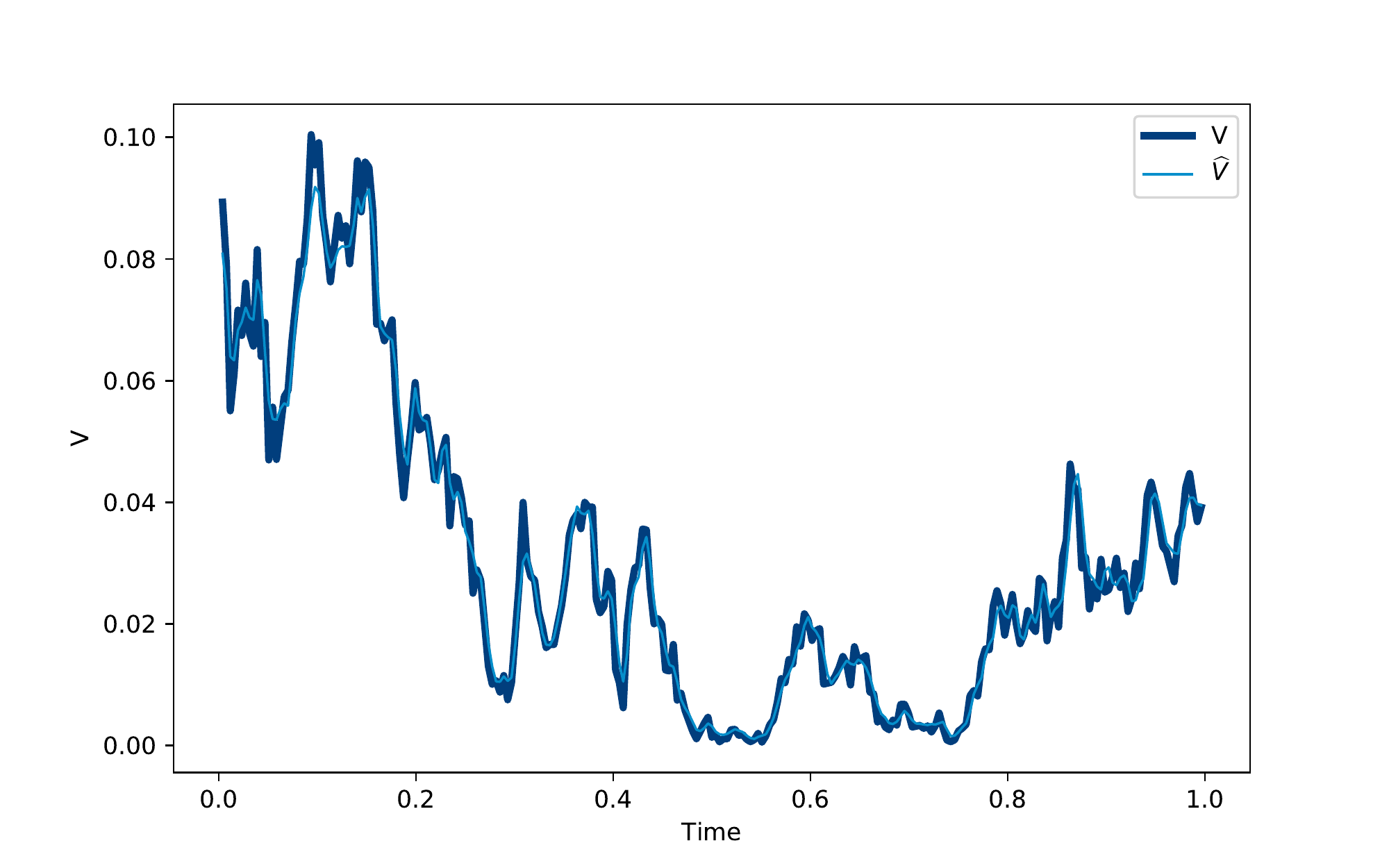}
  \caption{Heston}
  \label{fig:sub1}
\end{subfigure}%
\begin{subfigure}{.5\textwidth}
  \centering
  \includegraphics[width=\linewidth]{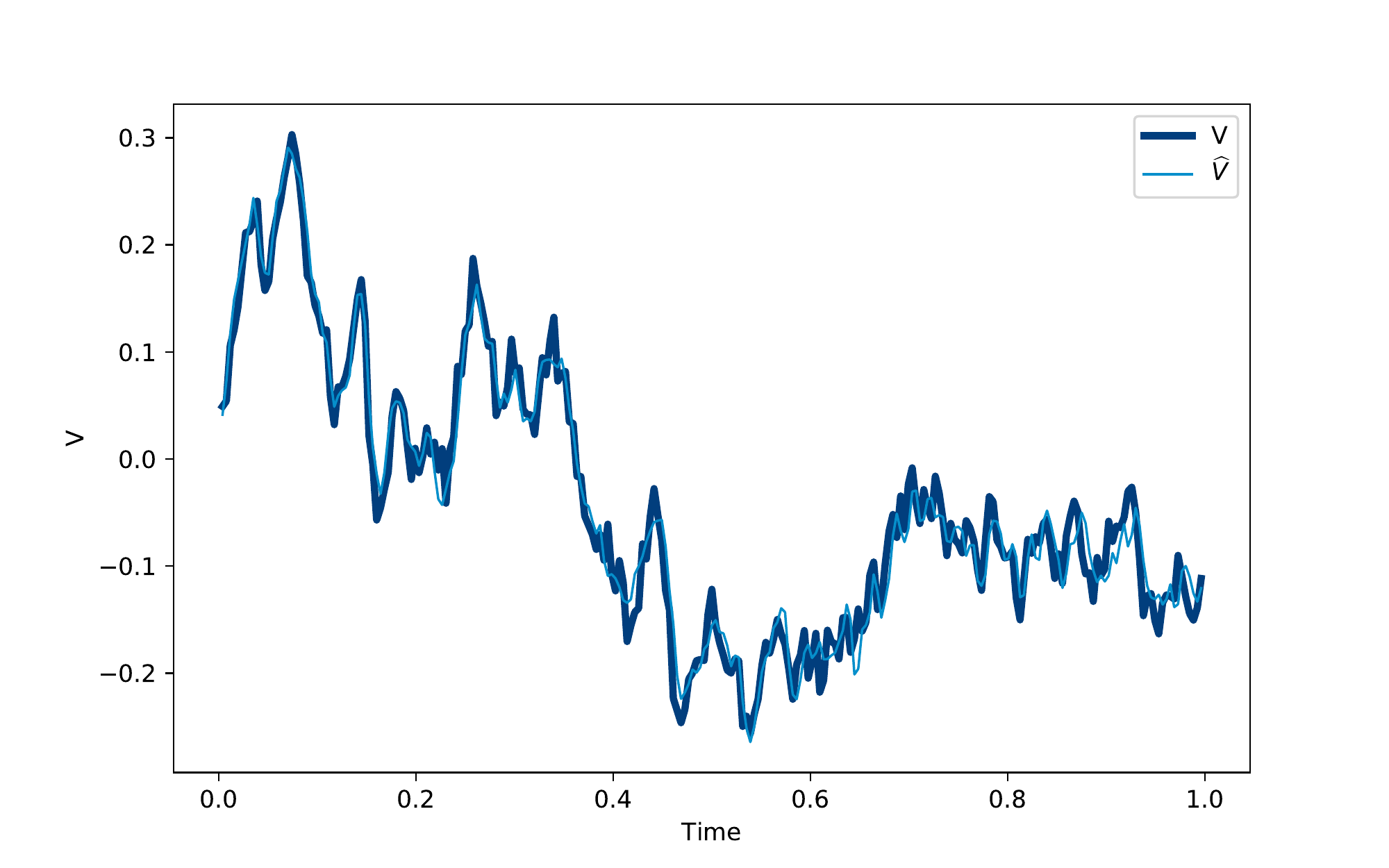}
  \caption{Exp-OU}
  \label{fig:sub2}
\end{subfigure}
\caption{Estimation of the volatility process $V$ using Cuchiero-Teichmann procedure.}
\label{fig:vol}
\end{figure}

\begin{figure}[h!]
\centering
\begin{subfigure}{.5\textwidth}
  \centering
  \includegraphics[width=\linewidth]{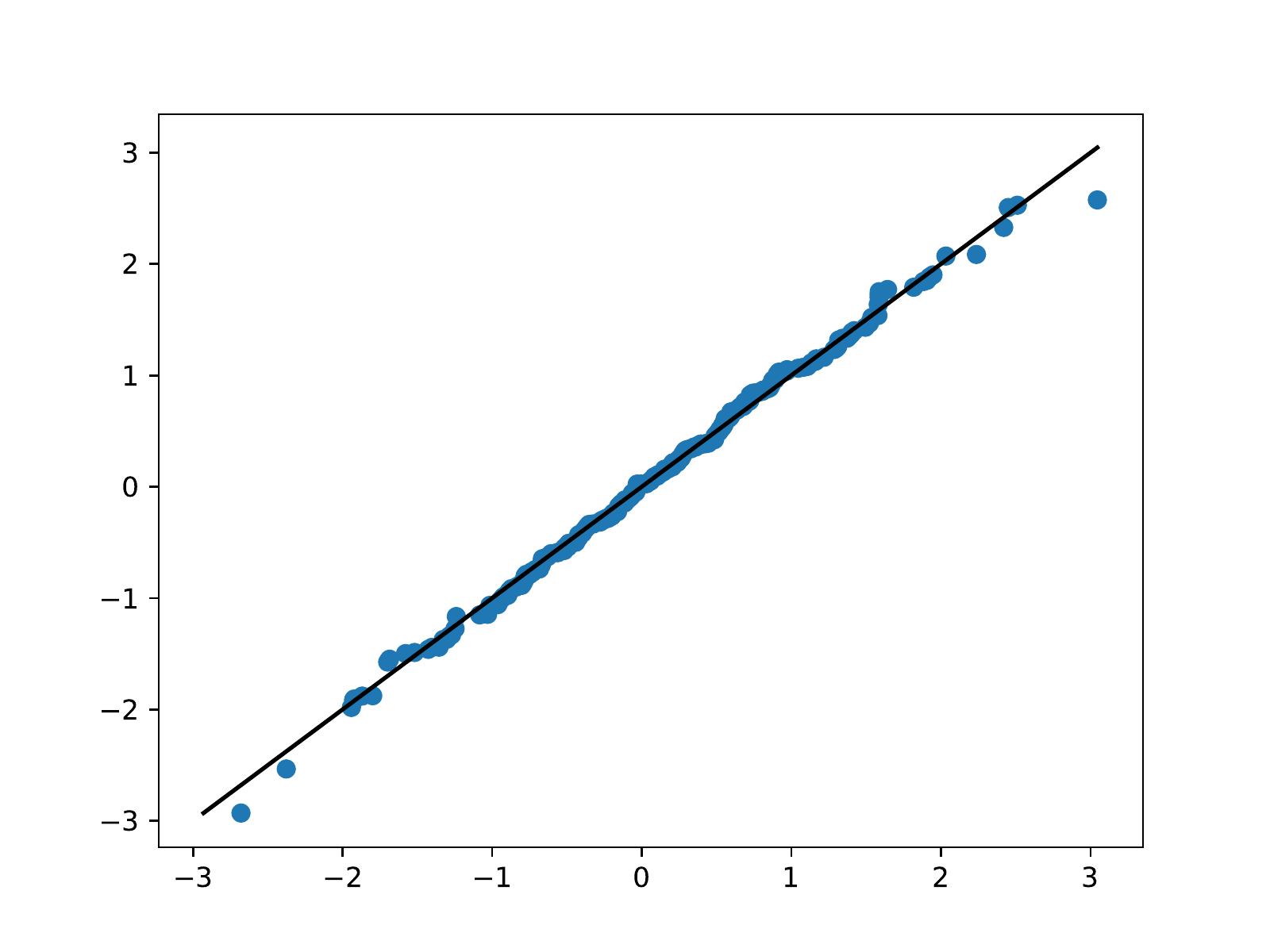}
  \caption{Heston}
  \label{fig:sub1}
\end{subfigure}%
\begin{subfigure}{.5\textwidth}
  \centering
  \includegraphics[width=\linewidth]{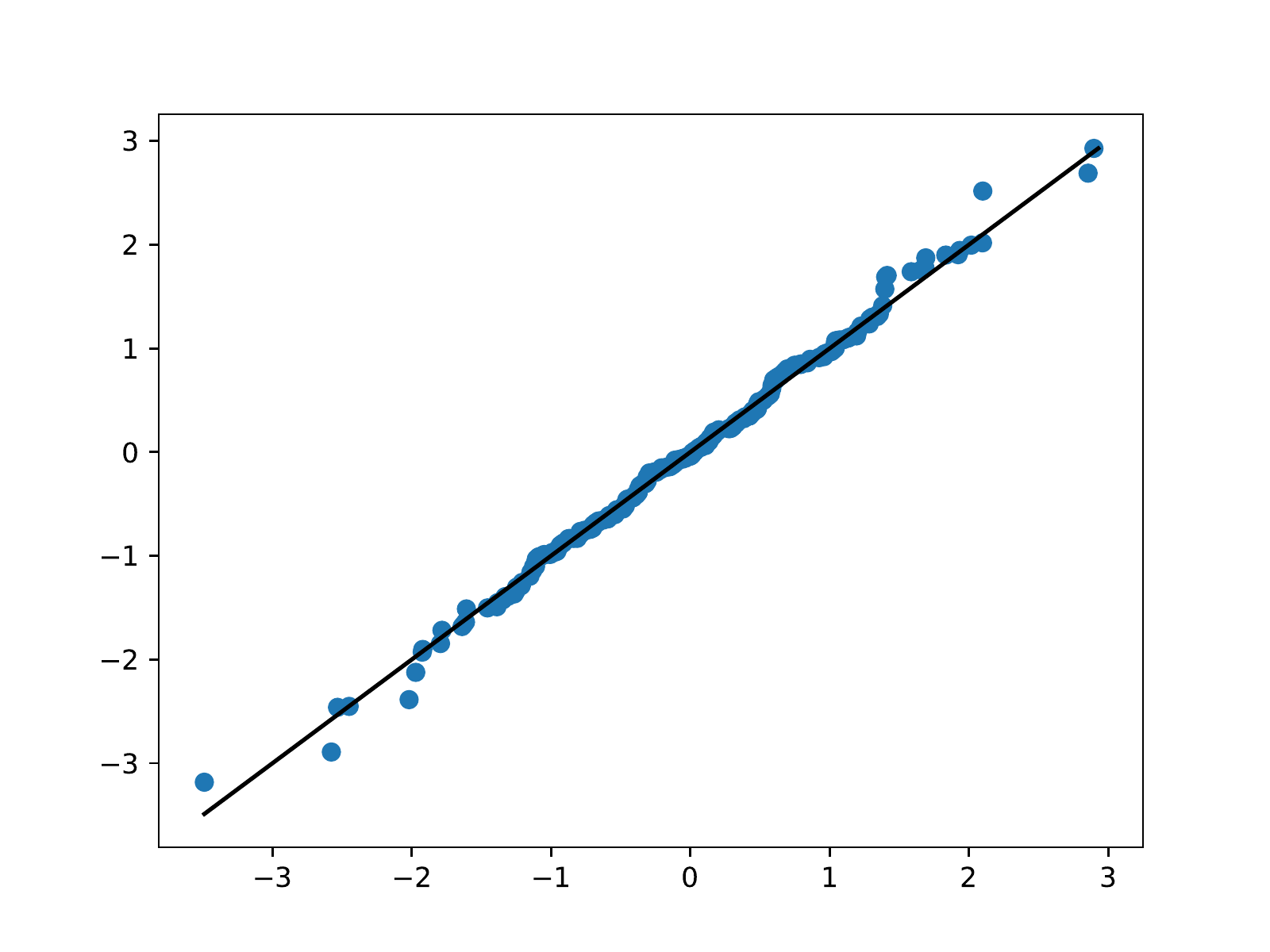}
  \caption{Exp-OU}
  \label{fig:sub2}
\end{subfigure}
\caption{QQ-plot of $\epsilon$ against $\hat{\epsilon}$ for both models setting the parameters as the true ones.}
\label{fig:qq_plot}
\end{figure}

\begin{figure}[h!]
\centering
\begin{subfigure}{.5\textwidth}
  \centering
  \includegraphics[width=\linewidth]{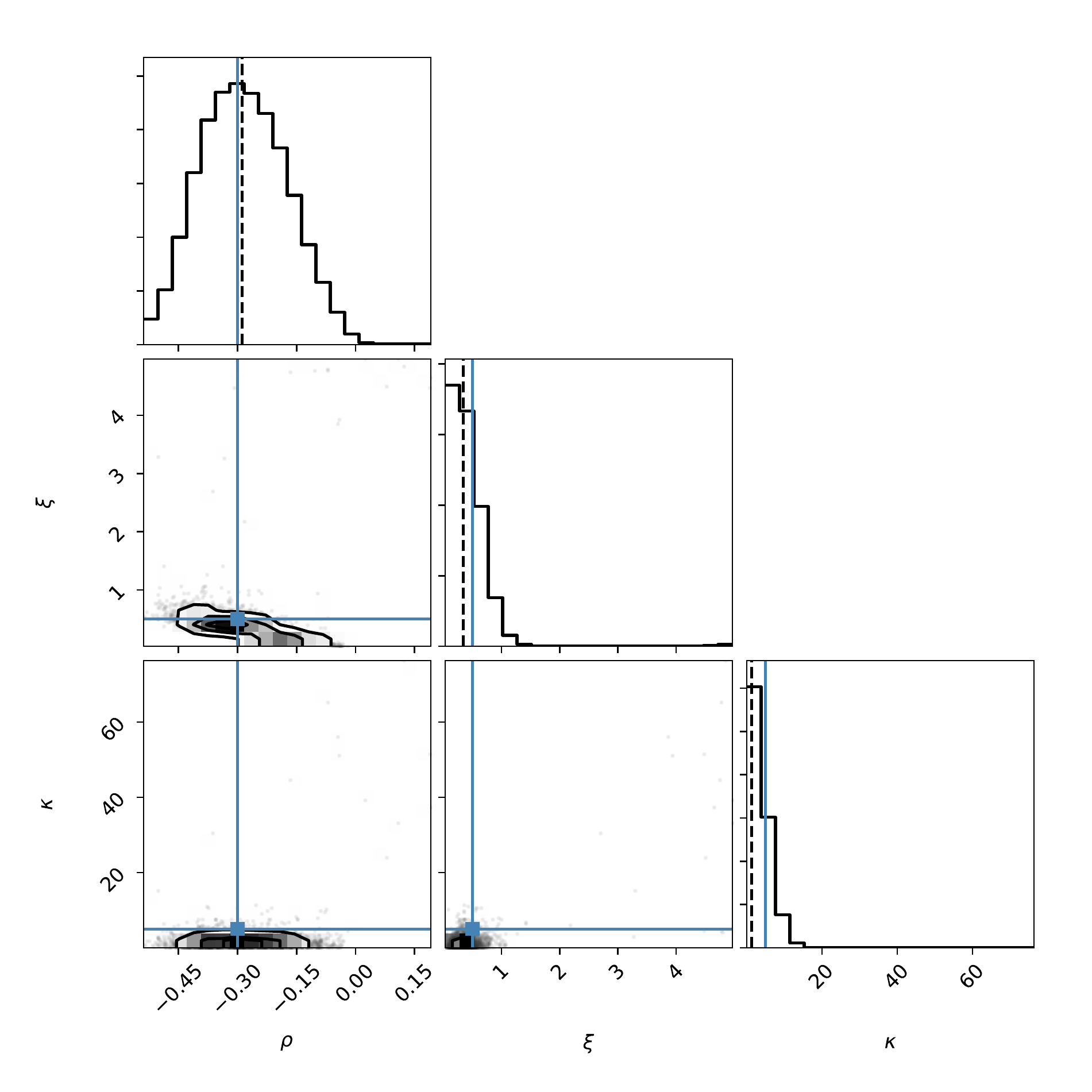}
  \caption{Heston}
  \label{fig:sub1}
\end{subfigure}%
\begin{subfigure}{.5\textwidth}
  \centering
  \includegraphics[width=\linewidth]{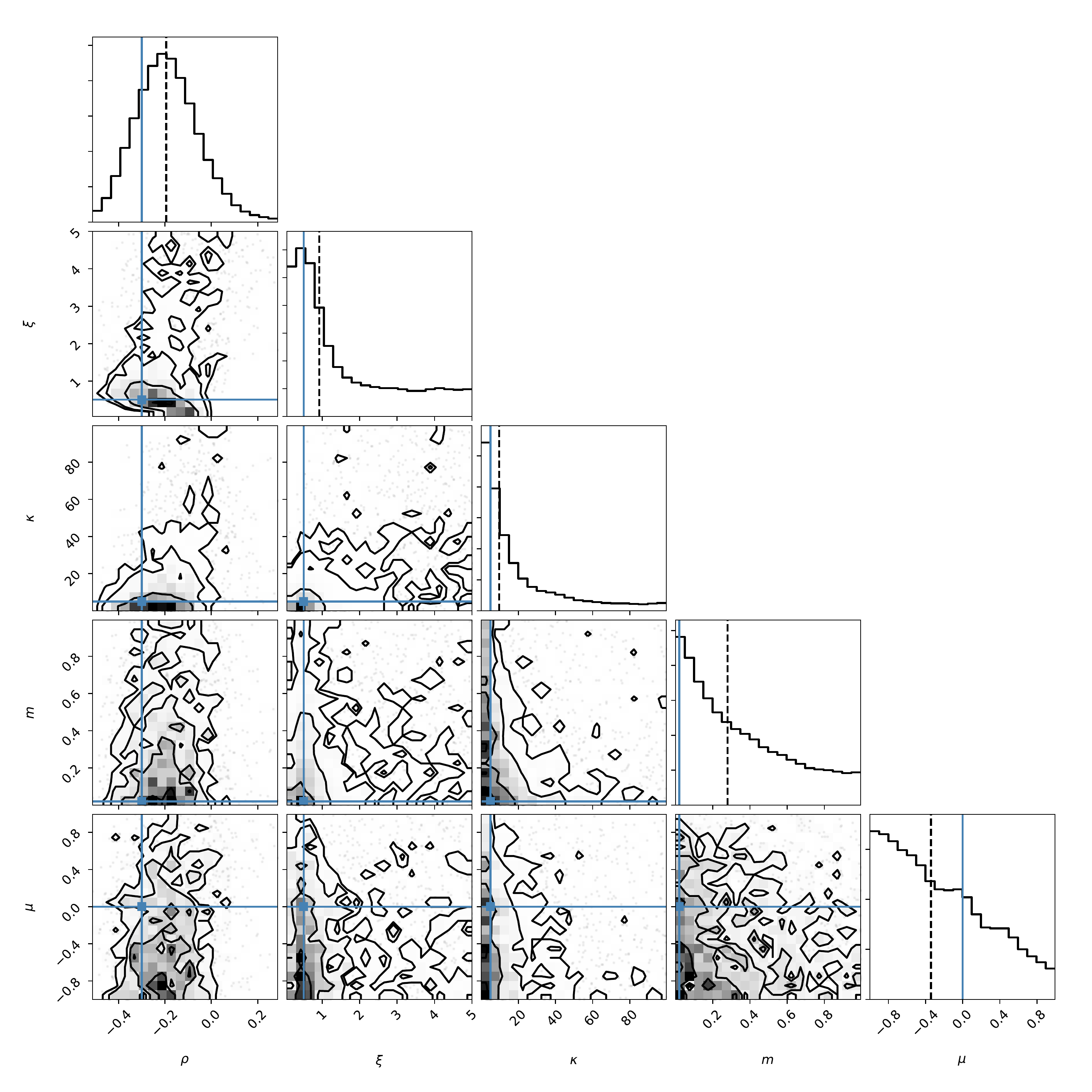}
  \caption{Exp-OU}
  \label{fig:sub2}
\end{subfigure}
\caption{Estimation of the identifiable parameters of each analyzed model.}
\label{fig:param}
\end{figure}

\noindent {\bf Acknowledgemens.} We express our gratitude to Josef Teichmann for sharing with us the code for calculation of the
estimator that we used in this paper. Milan Merkle
acknowledges the support by grants III 44006 and 174024 from Ministry of Education, Science and
Technological Development of Republic of Serbia. Yuri F. Saporito
acknowledges the support by grant 210.168/2017 from Funda\c{c}\~ao Carlos Chagas Filho de Amparo \`a Pesquisa do Estado do Rio de Janeiro.

\bibliographystyle{plain}

\appendix

\section{Identifiability}\label{sec:identifiability}

Given a stochastic model with a parameter $\theta \in \Theta$, let $\hat{\theta}$ be an estimator for $\theta$ based on a sample of size $n$. Let
 $\hat{\theta}\; |\; \theta_0$ denote the random element $\hat{\theta}$ under condition that the true value of the parameter $\theta$ equals $\theta_0$.
 It is said that $\theta$ is non-identifiable if for any $K\in \bN$
there is a sample of size $n\geq K$ and two distinct values $\theta_1,\theta_2\in \Theta$  such that the random elements
$\hat{\theta}\; |\; \theta_1$  and $\hat{\theta}\; |\; \theta_2$
have the same probability distribution.
This can be one among many possible  ways to define  the phenomenon that has been discussed
through statistical literature since long ago.
There is a considerable number of papers on the topic, especially in the econometric context, see \cite{malinvaud1980+} or \cite{hamuma2016+}
in the framework of Bayesian methods, and many others.

The parameter $\theta$ is said to be identifiable if it is not non-identifiable. Note that identifiability property is relative to the given estimator.
As shown in \cite{gabriel1978+}, the identifiability is a weaker property than consistency: if  $\hat{\theta}$ is consistent estimator of $\theta$, then
$\theta$ is identifiable; however $\theta$ can be identifiable even if there does not exist a consistent estimator.

\end{document}